\documentclass[12pt,leqno,a4paper]{amsart}

\usepackage{amssymb,enumerate, hyperref}
\usepackage{amsmath,amscd,amsthm, tikz-cd}
\usepackage{mathrsfs}
\usepackage{txfonts}
\usepackage{hyperref,cite}
\usepackage{xcolor}
\usepackage{framed}

\newcommand{\Aut}{\operatorname{Aut}}
\newcommand{\Out}{\operatorname{Out}}
\newcommand{\Irr}{\operatorname{Irr}}
\newcommand{\IBr}{\operatorname{IBr}}
\newcommand{\Lin}{\operatorname{Lin}}
\newcommand{\Res}{\operatorname{Res}}
\newcommand{\Ind}{\operatorname{Ind}}
\newcommand{\res}{\operatorname{res}}
\newcommand{\cores}{\operatorname{cor}}
\newcommand{\C}{\operatorname{C}}
\newcommand{\N}{\operatorname{N}}
\newcommand{\Z}{\operatorname{Z}}
\newcommand{\rO}{\operatorname{O}}
\newcommand{\Dec}{\operatorname{Dec}}
\newcommand{\GL}{\operatorname{GL}}

\newcommand{\bG}{\mathbf{G}}

\newcommand{\tbG}{\widetilde{\mathbf{G}}}
\newcommand{\FF}{\mathbb{F}}
\newcommand{\ZZ}{\mathbb{Z}}

\newcommand{\tB}{\widetilde{B}}
\newcommand{\cB}{\mathcal{B}}
\newcommand{\cC}{\mathcal{C}}
\newcommand{\cO}{\mathcal{O}}
\newcommand{\cT}{\mathcal{T}}
\newcommand{\fS}{\mathfrak{S}}
\newcommand{\cE}{\mathscr{E}}
\newcommand{\ty}[1]{\mathsf{#1}}

\let\vhi=\varphi
\let\ti=\times
\let\al=\alpha

\theoremstyle{theorem}
\newtheorem{thm}{Theorem}[section]
\newtheorem{lem}[thm]{Lemma}
\newtheorem{prop}[thm]{Proposition}
\newtheorem{cor}[thm]{Corollary}

\newtheorem{conj}[thm]{Conjecture}
\newtheorem{cond}[thm]{Condition}

\theoremstyle{definition}
\newtheorem{defn}[thm]{Definition}
\newtheorem{rmk}[thm]{Remark}
\newtheorem{example}[thm]{Example}
\newtheorem*{rem}{Remark}

\begin{document}

\title[Brauer $A(\infty)$ condition]{The Brauer $A(\infty)$ condition and Navarro's conjecture on Brauer characters under coprime actions}

\author{Zhicheng Feng}
\address[Z. Feng]{Shenzhen International Center for Mathematics and Department of Mathematics, Southern University of Science and Technology, Shenzhen 518055, China}
\makeatletter
\email{fengzc@sustech.edu.cn}
\makeatother

\author{Lin Wu}
\address[L. Wu]{Shenzhen International Center for Mathematics and Department of Mathematics, Southern University of Science and Technology, Shenzhen 518055, China}
\makeatletter
\email{12431015@mail.sustech.edu.cn}
\makeatother

\thanks{The authors gratefully acknowledge financial support by NSFC (12622101 and 12431001).}

\begin{abstract}
In this paper, we establish the Brauer $A(\infty)$ condition for all primes. Consequently, we show a conjecture of Navarro on coprime actions and Brauer characters for primes larger than 3.
\end{abstract}



\date{\today}

\maketitle

\section{Introduction}

For $G$ the universal covering of a simple group $S$ of Lie type, it is well-known that the outer automorphism group $\Out(G)$ decomposes as a semi-direct product $\Out(G)=\textup{Outdiag}(G)\rtimes D$ where $\textup{Outdiag}(G)$ is the group of outer diagonal automorphisms of $G$ (see, e.g., \cite[\S 2.5]{GLS98}) and $D$ is a group of field and graph automorphisms (denoted by $\Phi_G\Gamma_G$ in \cite{GLS98}). 

In 2012, Sp\"ath \cite{Sp12} suggested the following condition, which is called \emph{$A(\infty)$ condition} later by Cabanes and Sp\"ath (see Definition 2.2 of \cite{CS19}), in the investigation of inductive McKay condition for groups of Lie type.

\begin{cond}[$A(\infty)$ condition]
Assume that $S$ has no exceptional Schur multiplier.
For every $\chi\in\Irr(G)$, there exists a $\textup{Outdiag}(G)$-conjugate $\chi_0$ of $\chi$ such that the stabilizer of $\chi_0$ in $\Out(G)$ has the form $C'D'$, where $C'\le \textup{Outdiag}(G)$ and $D'\le D$, and such that $\chi_0$ extends to $GD'$.
\end{cond}

As in \cite{Sp25}, we leave aside the $17$ simple groups of Lie type with an exceptional Schur multiplier (see the list in \cite[Table~6.1.3]{GLS98}); their character tables, together with the action of automorphisms, are well known (see \cite{CCNPW85}).
We note that without this restriction, the $A(\infty)$ condition can fail; a counterexample is $S=\textup{PSL}_2(9)$, whose Schur multiplier has order $6$.

The $A(\infty)$ condition played an important role in the proof of the inductive McKay condition.
It was continuously studied by Cabanes and Späth in a series of papers \cite{CS17a,CS17b,CS19}, and ultimately proved by Sp\"ath \cite{Sp25} in 2025.
Combined with Lusztig's Jordan decomposition (see, e.g. \cite[\S2.6]{GM20}), we can use the $A(\infty)$ condition to completely determine the action of the automorphism group on the set of irreducible characters for quasi-simple groups of Lie type.

Now we consider the following analogue of the $A(\infty)$ condition for Brauer characters (see also Definition~\ref{defn:BrauerAinf} below).

\begin{cond}[Brauer $A(\infty)$ condition]\label{cond:1.1}
Assume that $S$ has no exceptional Schur multiplier.
For every $\varphi\in\IBr(G)$, there exists a $\textup{Outdiag}(G)$-conjugate $\varphi_0$ of $\varphi$ such that the stabilizer of $\varphi_0$ in $\Out(G)$ has the form $C'D'$, where $C'\le \textup{Outdiag}(G)$ and $D'\le D$, and such that $\varphi_0$ extends to $GD'$.
\end{cond}

The Brauer $A(\infty)$ condition plays a crucial role in the criteria for the inductive Alperin weight condition by Brough and Sp\"ath \cite{BS22} and for the inductive Brauer--Glauberman condition by Feng and Sp\"ath \cite{FS23}.
In the work \cite{FLZ22} of Feng, Z.~Li, and Zhang on the reduction of the inductive blockwise Alperin weight condition to quasi-isolated blocks, the Brauer $A(\infty)$ condition is also an assumption.

The Brauer $A(\infty)$ condition has been established for groups of type $\ty A$ by Feng, C.~Li and Zhang \cite{FLZ21}, for groups of type $\ty C$ and the prime 2 by Feng and Malle \cite{FM22}. These works depends on unitriangularity of decomposition matrices for related groups; in the unitriangular situation, there is a canonical bijection (see \cite[Lemma 2.1]{Na23}).
Unfortunately, unitriangularity has not been established for all groups of Lie type; in general, unitriangularity remains an open problem.

Recently, Feng, Z.~Li and Zhang \cite{FLZ26} proved the  Brauer $A(\infty)$ condition for groups of type $\ty B$, $\ty C$ or $\ty E_7$ at good non-defining primes.
The proofs in \cite{FLZ26} are independent of unitriangularity and rely only on the unimodularity of decomposition matrices.
In this paper, we establish the Brauer $A(\infty)$ condition.

\begin{thm}\label{mainthm-1}
The Brauer $A(\infty)$ condition holds.
\end{thm}

In the past, we often used the unitriangularity of decomposition matrices to obtain equivariance and extendibility, thereby reducing the problem to the ordinary characters in order to use the $A(\infty)$ condition. 
Since unitriangularity remains open at present, we turn instead to studying Brauer characters via ordinary characters without relying on unitriangularity.
For this, we calculate cohomology elements, and prove that the  Brauer $A(\infty)$ condition follows by $A(\infty)$ condition (see Theorem~\ref{thm:ord-Bra}).

Theorem~\ref{mainthm-1} allows us to remove Assumption~5.3 from Theorem~1 of \cite{FLZ22}, thereby completing the reduction of the inductive blockwise Alperin weight condition for groups of Lie type in non-defining characteristic to quasi-isolated blocks.

Let $A$ be a finite group acting corpimely by automorphisms on a finite group $G$. and write $\C_G(A)$ for the
fixed-point subgroup. 
The Glauberman--Isaacs correspondence (see \cite[\S13]{Is06}) provides
a canonical bijection between $\Irr_A(G)$, the set of
$A$-invariant irreducible complex characters of $G$, and
$\Irr(\C_G(A))$. 
This correspondence is a fundamental tool in the
representation theory of finite groups, with applications ranging
from the proofs of the McKay and Alperin weight conjectures for
$p$-solvable groups to the Langlands program. A consequence
is there exists an $A$-equivariant bijection between
$\Irr(G)$ and the set of conjugacy classes of $G$.

We use Theorem~\ref{mainthm-1} to investigate the following conjecture, proposed by Navarro, which is in some sense a modular analogue of the Glauberman--Isaacs correspondence.

\begin{conj}[Navarro, 1994 {\cite{Na94}}]\label{conj:coprime}
Suppose that $A$ and $G$ are finite groups such that $A$ acts coprimely via automorphisms on $G$.
Let $p$ be a prime.  Then
\[|\IBr_A(G)| = |\IBr(\C_G(A))|,\]
where $\IBr_A(G)$ denotes the set of irreducible $p$-Brauer characters of $G$ fixed by $A$.   
\end{conj}

Conjecture~\ref{conj:coprime} is equivalent to the existence of an 
$A$-equivariant bijection between the set of $p$-regular conjugacy classes of $G$ and the set of irreducible $p$-Brauer characters of $G$ (see Corollary~13.10 and Lemma~13.23 of \cite{Is06}).

This problem was first considered by Uno in~\cite{Un83}, where he gave a natural map $\IBr_A(G)\to\IBr(\C_G(A))$ when the group $G$ is $p$-solvable. 
Several other proofs of this correspondence were also given by Wolf~\cite{Wo87} and Graves~\cite{Gr94}.
Additionally, Navarro, Sp\"ath and Tiep \cite{NST17} proved that $A$ fixes a unique irreducible Brauer character of $G$ if and only if $\C_G(A)$ is an $p$-group.

However, when $G$ is not $p$-solvable, no
further construction of such a correspondence has been made.
Efforts were therefore made to reduce it to simple groups, especially after the McKay conjecture was successfully reduced to simple groups by Isaacs, Malle and Navarro~\cite{IMN07} and the Alperin weight conjecture by Navarro and Tiep~\cite{NT11}.
Here we mention the significant achievement of Cabanes and Sp\"ath \cite{CS26} in recently proving the McKay conjecture.
In~2016, Sp\"ath and Vallejo~\cite{SV16} reduced Conjecture~\ref{conj:coprime} to a problem on quasi-simple groups; they proved that  Conjecture~\ref{conj:coprime} holds for all finite group at the prime $p$ if the so-called inductive Brauer--Glauberman condition (see Definition~6.1 of \cite{SV16}) holds for all non-abelian simple groups at the prime $p$.

In \cite{NST17}, Navarro, Sp\"ath and Tiep proved that the inductive Brauer--Glauberman condition holds for all simple groups with cyclic outer automorphism groups, for all simple groups not of Lie type, for every finite simple
group of Lie type in defining characteristic.
Feng and Sp\"ath \cite{FS23} showed that all simple groups of type $\ty A$ satisfy the inductive Brauer--Glauberman condition.

In this paper, we establish the inductive Brauer--Glauberman condition for all simple groups of Lie type at all good primes (see Theorem~\ref{thm:good-primes}). Consequently, we prove:

\begin{thm}\label{mainthm-2}
If $p>3$, then Conjecture~\ref{conj:coprime} holds.    
\end{thm}

By the work of Feng--Sp\"ath \cite{FS23}, we need an equivariant bijection between a ``good" basic set and irreducible Brauer characters.
Precisely, we use the basic set for groups of Lie type in good non-defining primes by Geck and Hiss \cite{GH91,Ge93}.
We study the relationship between basic sets and Brauer characters without relying on unitriangularity.
If the outer automorphism groups are hypoelementary, then this follows from the result of Conlon (Theorem~\ref{thm:Colon}).
When the outer automorphism groups are not hypoelementary, we need to develop some results on permutation group theory (Theorem~\ref{thm:fixpt}).
The key crucial here is that if a finite group acts on two finite sets, then to what extent does the isomorphism of the corresponding permutation modules imply that the two actions are permutation isomorphic?
In this sense, this paper is a continuation of the paper \cite{FLZ26}.

This paper is organized as follows. 
In Section~\ref{sec:2}, we recall the theorem of Conlon (Theorem~\ref{thm:Colon}) and establish some results on permutation group theory.
Section~\ref{sec:extend} provides some properties of extensions of Brauer characters, while Section~\ref{sec:A(infty)} establishes the Brauer $A(\infty)$ condition.
In Section~\ref{sec:size-series}, we prove Theorem~\ref{mainthm-1}, estimate upper bounds on the size of Lusztig series for groups of types $\ty D_4$ and $\ty E_6$ and complete the proof Theorem~\ref{mainthm-2}.

\section{Some results of permutation group theory}\label{sec:2}

Let $G$ be a group acting on a set $\Omega$. 
For a subgroup $H\le G$ (resp. an element $g \in G$), we denote the set of fixed points of $H$ in $\Omega$ by $\Omega^H$.
We abbreviate $\Omega^g$ for $\Omega^{\langle g\rangle}$.

Let \( G \) be a group that acts on sets $\Omega$ and $\Lambda$.
We say that $\Omega$ and $\Lambda$ are isomorphic as $G$-sets if there exists a $G$-equivariant bijection between them.

The following lemma is elementary.

\begin{lem}\label{lem-fix-pt-quo}
	Let $G$ be a group that acts on a set $\Omega$, and $N \trianglelefteq G$ a normal subgroup. Then
	\[\Omega^G =  (\Omega^N)^{G/N}.\]
\end{lem}

We will make use of the following results of Burnside.

\begin{lem}[Burnside {\cite[Thm.~2.3.2]{Bo00}}]\label{lem:Burnside}
	Let \( G \) be a group that acts on finite sets $\Omega$ and $\Lambda$. 
	Then $\Omega$ and $\Lambda$ are isomorphic as $G$-sets if and only if 
	$|\Omega^H|=|\Lambda^H|$ for each subgroup $H$ of \( G \). 
\end{lem}

\begin{defn}\label{defn:equiv-mtx}
	Let \( G \) be a finite group acting on two finite sets \( \Omega \) and \( \Lambda \) of the same cardinality, and let \( \cO \) be a commutative ring with unity. 
	Let \( U = (u_{\omega,\lambda})_{\omega \in \Omega,\ \lambda \in \Lambda} \) be a matrix over \( \cO \) with rows indexed by \( \Omega \) and columns indexed by \( \Lambda \).
	We say that \( U \) is a \emph{\( G \)-invariant matrix between \( \Omega \) and \( \Lambda \)} if
	\[
	u_{g\cdot \omega,\, g\cdot \lambda} = u_{\omega,\lambda}
	\]
	for all \( g \in G \), \( \omega \in \Omega \), and \( \lambda \in \Lambda \).
\end{defn}

If the matrix \( U \) in Definition~\ref{defn:equiv-mtx} has determinant invertible in \( \cO \), then there is an isomorphism of \( \cO G \)-modules between \( \cO \Omega \) and \( \cO \Lambda \).

We recall the definition of finite ($\ell$-)hypoelementary groups (see for instance \cite[Definition~5.5.3]{Be98}).

\begin{defn}
Let $\ell$ be a prime.  A finite group $G$ is called \emph{$\ell$-hypoelementary} if it has a normal
$\ell$-subgroup $P\trianglelefteq G$ such that $G/P$ is cyclic of order prime to $\ell$.
In this situation, $P=\rO_\ell(G)$ is a Sylow $\ell$-subgroup of $G$.

A finite group is called \emph{hypoelementary} if it is $\ell$-hypoelementary for some prime
$\ell$.
\end{defn}

Clearly, subgroups and quotient groups of $\ell$-hypoelementary groups are also $\ell$-hypoelementary.

We recall the following result of Conlon. 

\begin{thm}[Conlon \cite{Co68}]\label{thm:Colon}
Let a finite group $G$ act on finite sets $\Omega$ and $\Lambda$, where $|\Omega|=|\Lambda|$. Let $\ell$ be a prime.
Then the following are equivalent.
\begin{enumerate}[\rm(1)]
	\item There exists a $G$-invariant integer matrix $U$ between $\Omega$ and $\Lambda$, 
	and $\ell\nmid\det U$.
	\item $\FF_\ell \Omega\cong\FF_\ell \Lambda$.
	\item $|\Omega^H|=|\Lambda^H|$, for every $\ell$-hypoelementary subgroup $H$ of $G$.
\end{enumerate}
\end{thm}

\begin{proof}
See  \cite[Thm.~3.5.5]{Bo00}.
\end{proof}

\begin{cor}\label{cor:hypoelementary}
	Let a finite $\ell$-hypoelementary group $G$ act on finite sets $\Omega$ and $\Lambda$, where
	$|\Omega|=|\Lambda|$. 
	Suppose that there is a $G$-invariant integer matrix $U$ between $\Omega$ and $\Lambda$, 
	and $\ell\nmid\det U$. Then
	\(	\Omega\) and \(\Lambda\) are isomorphic as $G$-sets.
\end{cor}

\begin{proof}
This follows from Lemma~\ref{lem:Burnside} and Theorem~\ref{thm:Colon} directly, since every subgroup of an $\ell$-hypoelementary group is also $\ell$-hypoelementary. 
\end{proof}

We use double curly brackets to denote multiset, i.e. $\{\{a, b, b\}\}$.

\begin{lem}\label{lem: divisor}
Let $p$ be a prime number and let $X, Y$ be finite $G$-sets. Suppose the sizes of orbits of $X$ and $Y$ are $\{d_i\}_{i= 1}^r$ and $\{e_j\}_{j= 1}^s$, respectively. If $\mathbb{F}_pX\cong \mathbb{F}_pY$ as $\mathbb{F}_pG$-modules, then $r= s$ and $\{\{v_p(d_1), \cdots, v_p(d_r)\}\}= \{\{v_p(e_1), \cdots, v_p(e_s)\}\}$.
\end{lem}

\begin{proof}
Follows from \cite[Cor.~5.3]{Is01} directly.
\end{proof}

\begin{thm}\label{thm:fixpt}
Let a finite group $G$ act on finite sets $\Omega$ and $\Lambda$.
Assume that $|\Omega^E|=|\Lambda^E|$ for every hypoelementary subgroup $E$ of $G$.
If $|\Omega|\le 23$, then $\Omega$ and $\Lambda$ are isomorphic as $G$-sets.
\end{thm}    

\begin{proof}

Let $G$ be a minimal counterexample. Then $G\neq 1$ and for any proper subgroup $H< G$, $\Omega\cong \Lambda$ as $H$-sets. By Lemma \ref{lem:Burnside}, we have $|\Omega^G|\neq |\Lambda^G|$. Let $\{G_i\}_{i= 1}^z$ be the representatives of the conjugacy classes of subgroups of $G$ and $G_1= G$. Then we have orbit decompositions $\Omega\cong \bigsqcup_{i= 1}^zn_i(G/ G_i)$ and $\Lambda\cong \bigsqcup_{i= 1}^zm_i(G/ G_i)$. We may assume $\min\{n_i, m_i\}= 0$ and $n_1> 0$. Write $n= |\Omega|= |\Lambda|$. Let $H$ be a subgroup of $G$ and let $j\in \{1,\cdots z\}$ such that $H$ is conjugate to $G_j$. Then define $m_H(\Omega)= n_j, m_H(\Lambda)= m_j$. 

Now let $M$ be a maximal subgroup of $G$.

Claim 1: $m_M(\Omega)= 0$. If not, there exists $i\in \{1, 2, \cdots, z\}$ such that $M$ conjugate to $G_i$ and $n_i\neq 0$. Thus $m_i= 0$. Note that $|\Omega^{G_i}|= |\Lambda^{G_i}|> 0$, thus there exists $\lambda\in \Lambda$ such that $G_i\leq G_\lambda$. Since $m_i= 0$, we get $G_\lambda= G$, which forces $m_1\neq 0$. Hence we get a contradiction.

Claim 2: $m_M(\Lambda)= \frac{n_1}{[\N_G(M): M]}$. Similarly, we assume $M$ is conjugate to $G_i$ for some $i$. Since $m_1= 0$, we have $|\Lambda^M|= \#\{\lambda\in \Lambda\mid G_\lambda= M\}$. If $\lambda$ is a point fixed by $M$, then the orbit $G\cdot\lambda$ is isomorphic to $G/ M\cong G/ G_i$. Since $gM\in (G/ M)^M\Leftrightarrow g\in \N_G(M)$, we get $|(G/ M)^M|= [\N_G(M): M]$. So $|\Lambda^M|= m_M(\Lambda)[\N_G(M): M]$. By claim 1, we get $|\Omega^M|= |\Omega^G|= n_1$. Combining the above equalities, we get $m_M(\Lambda)= \frac{n_1}{[\N_G(M): M]}$.

The above two claims tell us that maximal subgroup of $G$ is not the stabilizer of any point in $\Omega$. In particular, there is no orbit in $\Omega$ has prime size. 

By Theorem \ref{thm:Colon}, $\mathbb{F}_p\Omega\cong \mathbb{F}_p\Lambda$ as $\mathbb{F}_pG$-modules for any prime number $p$. Suppose the sizes of orbits of $\Omega$ and $\Lambda$ are $\{d_i\}_{i= 1}^r$ and $\{e_j\}_{j= 1}^s$, respectively. By Lemma \ref{lem: divisor}, we have $r= s$ and $\{\{v_p(d_1), \cdots, v_p(d_r)\}\}= \{\{v_p(e_1),\cdots, v_p(e_r)\}\}$ for any prime number $p$.

We define two functions: For any positive integer $m= p_1^{e_1}\cdots p_t^{e_t}$, define $\omega(m)= t, \sigma(m)= p_1^{e_1}+\cdots + p_t^{e_t}$ and $\omega(1)= \sigma(1)= 0$. Now $\sum_{i= 1}^r\omega(d_i)= \sum_{i= 1}^r\omega(e_i)$ and $\sum_{i= 1}^r\sigma(d_i)= \sum_{i= 1}^r\sigma(e_i)$. Since for any $i\in \{1,\cdots, r\}$ we have $d_i, e_i\leq 23< 2\cdot 3\cdot 5$, we see that $\omega(d_i), \omega(e_i)\leq 2$. Thus
\begin{equation*}
\begin{aligned}
\sum_i\omega(d_i)&= r- n_1+ \#\{1\leq i\leq r\mid \omega(d_i)= 2\} =: r- n_1+ u\\
\sum_i\omega(e_i)&= r+ \#\{1\leq i\leq r\mid \omega(e_i)= 2\} =: r+ v.
\end{aligned}
\end{equation*}
Then $u- v= n_1$. By direct calculation, $\sigma(x)\leq x$ for any $1\leq x\leq 23$.

Claim: $v\geq 1$. Suppose $v= 0$, then $\sum_i\sigma(e_i)= n$. But $\sum_i\sigma(e_i)= \sum_i\sigma(d_i)\leq \sum_{d_i> 1}d_i= n- n_1< n$, which is a contradiction.

Since $u- v= n_1$, we get $u\geq n_1+ 1$. Therefore $23\geq n\geq n_1+ 6u\geq 7n_1+ 6$, where the second inequality holds because the smallest positive integer with $2$ prime divisors is $6$. Thus $n_1\in \{1, 2\}$.

Case 1: $n_1= 2$. In this case, $3\leq u\leq \lfloor\frac{23- 2}{6}\rfloor= 3$. So $u= 3, v= 1$. If there is an orbit $O$ in $\Omega$ such that $|O|$ has $2$ prime divisors and $|O|> 6$, then its size is at least $10$ and $n\geq n_1+ 2\cdot 6+ 10= 24$, which is impossible. Deleting trivial orbits and the three orbits of size $6$, there are at most $23- 3\cdot 6- 2= 3$ points. But $\Omega$ has no orbit of prime size, thus $\{\{d_i\}\}_{i= 1}^r= \{\{1, 1, 6, 6, 6\}\}$. By direct calculation, we get $\{\{e_i\}\}_{i= 1}^r= \{\{6, 2, 2, 3, 3\}\}$. But $6+ 2+ 2+ 3+ 3\neq 20$, which is a contradiction.

Case 2: $n_1= 1$. If there is a normal maximal subgroup $M\leq G$ of prime index $q$, then $m_M(\Lambda)= \frac{1}{q}\in\mathbb{Z}$, which is impossible. So $G$ must be perfect.

Let $H< G$ be a proper subgroup. Suppose $[G: H]\leq 4$, then there is a nontrivial homomorphism $\psi: G\rightarrow \fS_4$. So $G/ \ker(\psi)$ is solvable and perfect, which implies $G= \ker(\psi)$ and $G= H$, a contradiction. Now we have $[G: H]\geq 5$.

Choose a nontrivial orbit in $\Omega$, say, isomorphic to $G/ H$ for some proper subgroup $H$. Since $H$ can not be maximal, there exists a maximal subgroup $M$ such that $H< M< G$. Thus $[G: H]\geq 5\cdot 2= 10$. This means the size of any nontrivial orbit in $\Omega$ is at least $10$. Since $23\geq n\geq 1+ 10(r- 1)$, we get $2\leq v+ 1= u\leq r-1\leq 2$.  So $r= 3, u= 2, v= 1$. Let $d_2, d_3$ be the sizes of the two nontrivial orbits in $\Omega$, then $(d_2, d_3)\in \{(10, 10), (10, 12)\}$. Because $\{\{v_2(d_i)\}\}_{i= 1}^r= \{\{0, 1, 1\}\}$ or $\{\{0, 1, 2\}\}$, there exists an orbit in $\Lambda$ of even size and its size is a prime power since $v= 1$. So its size is $2$ or $4$. Now we get a proper subgroup of index smaller than $5$, which is a contradiction.
\end{proof}

\begin{rem}
We mention that the bound 23 in Theorem~\ref{thm:fixpt} cannot be improved to 24.
In fact, the example in \cite[p.321]{We07} satisfies the hypotheses of Theorem~\ref{thm:fixpt};
here $G$ is a dihedral group of order 12, $|\Omega|=|\Lambda|=24$, but $\Omega$ and $\Lambda$ are not isomorphic as $G$-sets.
\end{rem}

\begin{lem}\label{lem:4.2-5}
Let a finite group $G$ act on finite sets $\Omega$ and $\Lambda$ such that $|\Omega^E|=|\Lambda^E|$ for every hypoelementary subgroup $E$ of $G$.
Assume that there exists a normal subgroup $N$ of $G$ such that $N$ is of prime order and $G/N$ is hypoelementary.
If $|\Omega^N|\le 5$, then $\Omega$ and $\Lambda$ are isomorphic as $G$-sets. \end{lem}

\begin{proof}
By Lemma~\ref{lem:Burnside}, it suffices to show that $|\Omega^H|=|\Lambda^H|$ for each subgroup $H$ of \( G \). 

Let $H\le G$. If $N\nsubseteq H$, then $H\cap N=1$ and thus $H\cong HN/N\le G/N$ is hypoelementary.
Hence we assume that $N\subseteq H$.
We regard $\Omega^N$ and $\Lambda^N$ as $H/N$-sets.

Now we let $L$ be a subgroup of $H$ containing $N$ such that $L/N$ is cyclic.
Then $L$ is hypoelementary.
By Lemma~\ref{lem-fix-pt-quo}, 
\[|(\Omega^N)^{L/N}|=|\Omega^L|=|\Lambda^L|=|(\Lambda^N)^{L/N}|.\]
Since $|\Omega^N|\le 5$,
\cite[Lemma~2.7]{FLZ26} implies that $\Omega^N$
and $\Lambda^N$ are isomorphic as $H/N$-sets.
In particular, \[|\Omega^H|=|(\Omega^N)^{H/N}|=|(\Lambda^N)^{H/N}|=|\Lambda^H|.\]  
\end{proof}
 
Now we consider modular representation with respect to the prime $p$.      
Let $G$ be a finite group, and denote by $G_{\!p'}$ the set of $p$-regular elements of $G$.
For $\chi\in\Irr(G)$,
we denote the restriction of $\chi$ to $G_{\!p'}$ by $\chi^0$.
The \emph{($p$-)decomposition matrix} is $(d_{\chi,\phi})_{\chi\in\Irr(G),\,\phi\in\IBr(G)}$ defined by \[
\chi^0=\sum_{\phi\in\IBr(G)}d_{\chi,\,\phi}\,\phi\] as a class function on $G_{\!p'}$ for any $\chi\in\Irr(G)$.
For $\cB\subseteq\Irr(G)$ and $\cC\subseteq\IBr(G)$ we set \[\Dec_{\cB,\,\cC}:=(d_{\chi,\,\phi})_{\chi\in\cB,\,\phi\in\cC}.\] 
If $\Dec_{\cB,\,\cC}$ is unimodular, then we say that $\cB$ is \emph{basic}.

For a finite group $X$, then we denote by $\Lin(X)$ the group of linear characters of $X$, and by $\Lin_{p'}(X)$ the Hall $p'$-subgroup of $\Lin(X)$.
By construction, $\Lin_{p'}(X)$ can be identical to group of linear Brauer characters of $X$, and 
acts on $\Irr(X)$ and $\IBr(X)$ by multiplication. 

\begin{cor}\label{cor:bas-set-equ}
Let $A$ be a finite group, and $N\le G$ be normal subgroups of $A$ such that $G/N$ is abelian.
Let $H$ be a subgroup of $\Lin_{p'}(G/N)\rtimes (A/G)$.
Suppose that $\cB\subseteq\Irr(G)$ and $\cC\subseteq\IBr(G)$ are $H$-stable sets such that \(\Dec_{\cB,\,\cC}\) is square and unimodular.
Assume that one of the following holds.
\begin{enumerate}[\rm(1)]
\item $H$ is hypoelementary.
\item $|\cB|\le 23$.
\item There is a normal subgroup $K$ of $H$ such that $K$ is of prime order, $H/K$ is hypoelementary and $|\cB^K|\le 5$.
\end{enumerate}
Then there exists a $H$-equivariant bijection between $\cB$ and $\cC$.
\end{cor}

\begin{proof}
First note that \(\Dec_{\cB,\,\cC}\) is a $H$-invariant integer matrix $U$ between  $\cB$ and $\cC$.
As \(\Dec_{\cB,\,\cC}\) is unimodular, it has determinant not divisible by any prime.
By Theorem~\ref{thm:Colon}, every hypoelementary subgroup of $H$ has equal numbers of fixed points in $\cB$ and $\cC$.
Therefore, this assertion follows by Corollary~\ref{cor:hypoelementary}, Theorem~\ref{thm:fixpt} and Lemma~\ref{lem:4.2-5} immediately.
\end{proof}

\section{Extendibility of ordinary characters and Brauer characters}\label{sec:extend}

In this section, we will study properties of Brauer character extendibility via ordinary character extendibility.
Some results in this section are of independent interest.

We first recall the terminology of extendibility of representations and the projective representations.
See for instance \cite[Chap.~8]{Na98} for more details.
Let $k$ be an algebraically closed field of characteristic $p$.
Let $G\trianglelefteq A$ be finite groups and let $\vhi\in\IBr(G)$ be $A$-invariant.
Then there exists a projective representation \[\mathcal P:A\to\GL_{\vhi(1)}(k)\] \label{def:proj}
such that the restriction $\mathcal P|_G$ affords $\vhi$ and $\mathcal P(gx)=\mathcal P(g)\mathcal P(x)$ and $\mathcal P(xg)=\mathcal P(x)\mathcal P(g)$ for all $g\in G$ and $x\in A$.
For every $x,y\in A$, \[\mathcal P(xy)=\al_\vhi(x,y)\mathcal P(x)\mathcal P(y),\] where $\al_\vhi:A/G\ti A/G\to k^\ti$ is a 2-cocycle.
The quotient of $2$-cocycles by $2$-coboundaries is the second cohomology group
$\textup{H}^2(A/G,k^\ti)$ (see, e.g., \cite[p.166]{Na98}), and the class of $\alpha_\vhi$ is denoted by $c_{A}(\vhi)$. \label{notation-cohomology}
We denote by $o_A(\vhi)$ the order of $c_{A}(\vhi)\in \textup{H}^2(A/G,k^\ti)$.
Recall that $\vhi$ extends to $A$ if and only if $o_{A}(\vhi)=1$ (see, e.g., \cite[Thm.~(8.15)]{Na98}).

Now we consider the extensions of Brauer characters.
Let us start from the case of one character and one Brauer character.

\begin{example}\label{example}
Let $G\trianglelefteq A$, let $\chi\in\Irr(G)$ and
$\varphi\in\IBr(G)$ be $A$-invariant.
If $\chi$ extends to $A$ and \(d_{\chi,\varphi}=1\), then $\varphi$ extends to $A$.
Below we give a sketch of the proof.
Let $\widetilde\chi\in\Irr(A)$ be an extension of $\chi$. Write	\(	\widetilde\chi^{\,0}
=\sum_{\eta\in\IBr(A)}
d_{\widetilde\chi,\eta}\eta.	\)
For $\eta\in\IBr(A)$, write
\(	\Res^A_G\eta
=\sum_{\psi\in\IBr(G)}e_{\eta,\psi}\psi.\)
Here all $d_{\widetilde\chi,\eta}$ and $e_{\eta,\psi}$ are
nonnegative integers. Restricting the preceding Brauer decomposition to
$G$ and comparing the coefficient of $\varphi$ gives
\(	1=d_{\chi,\varphi}
=\sum_{\eta\in\IBr(A)}
d_{\widetilde\chi,\eta}e_{\eta,\varphi}.	\)
Consequently, there exists $\eta_0\in\IBr(A)$ such that
\(	d_{\widetilde\chi,\eta_0}
=e_{\eta_0,\varphi}=1.	\)
In particular, $\Res^{A}_{G}\eta_0$ is multiplicity-free, whence $\Res^{A}_{G}\eta_0=\vhi$.
So $\eta_0$ is an extension of $\varphi$ to $A$.
\end{example}

In fact, Example~\ref{example} is a consequence of the following result of Schmid.

\begin{lem}[Schmid]\label{lem:Schmid}
	Let $G\trianglelefteq A$ be finite groups, let $\chi\in\Irr(G)$ and
	$\varphi\in\IBr(G)$ be $A$-invariant such that	\(d_{\chi,\varphi}\ne 0.\)
	If $\chi$ extends to $A$, then $o_{A}(\vhi)\mid d_{\chi,\varphi}$.
\end{lem}

\begin{proof}
This follows from the proof of \cite[Thm.~3]{Sc80}.
\end{proof}

\begin{cor}\label{cor:ext-dec}
Let $G\trianglelefteq A$ be finite groups, let $\chi\in\Irr(G)$ and $\varphi\in\IBr(G)$ be $A$-invariant such that \(d_{\chi,\varphi}\ne 0.\)
If $\gcd(|A/G|_{p'},\, |G|\vhi(1), \,d_{\chi,\varphi})=1$ and $\chi$ extends to $A$, then $\vhi$ extends to $A$.
\end{cor}

\begin{proof}
By \cite[Thm.~(8.29)]{Na98}, it suffices to prove that $\vhi$ extends to $X_\ell$ for every prime $\ell$ dividing $|A/G|$.
Here $X_\ell/G$ is a Sylow $\ell$-subgroup of $A/G$.
By a result of Green \cite[Thm.~(8.11)]{Na98}, $\vhi$ extends to $X_p$.
Therefore, we may assume that $A/G$ is an $\ell$-group with $\ell\ne p$.
Then this assertion follows by Lemma~\ref{lem:Schmid},  \cite[Chap.~3, Thm.~5.7]{NT89} and \cite[Chap.~2, Thm.~7.3]{NT89}.
\end{proof}

Recall that an element of $\ZZ\Irr(G)$ is called a \emph{generalized character} of $G$, and an element of $\ZZ\IBr(G)$ is called a \emph{generalized Brauer character} of $G$.

The following lemma is a slight generalization of \cite[(2.4)]{Ga81} and \cite[Thm.~6.15]{Da70}.

\begin{lem}\label{lem:Gagola}
Let $G\trianglelefteq A$ be finite groups such that $A/G$ is an $\ell$-group for some prime $\ell$.
If $\vhi\in\IBr(G)$ is $A$-invariant, then there exists an $A$-invariant character $\chi\in\Irr(G)$ such that $\ell\nmid d_{\chi,\vhi}$. 
\end{lem}

\begin{proof}
The proof needs only slight modifications on that of \cite[(2.4)]{Ga81}.
The case $p\ne\ell$ also follows by \cite[Thm.~6.15]{Da70}.
For completeness, we still give the details, nevertheless.	
Let $\widehat\vhi$ be the class function on $G$ with $\widehat\vhi(g)=\vhi(g_{p'})$ where $g_{p'}$ is the $p'$-part of $g$.
By \cite[Lemma~(2.15)]{Na98}, $\widehat\vhi$ is a generalized character of $G$, that is,  $\widehat\vhi=\sum\limits_{\chi\in\Irr(G)}a_\chi\chi$ where $a_\chi$ are integers.
As $\vhi$ is $A$-invariant, $\widehat\vhi$ is also $A$-invariant.
Hence $a_\chi$ are constant on $A$-orbits in $\Irr(G)$.
Also $d_{\chi,\vhi}$ are constant on $A$-orbits in $\Irr(G)$.

Restricting $\widehat\vhi$ to $p$-regular elements gives 
\[
\vhi=\sum\limits_{\chi\in\Irr(G)} a_\chi\sum_{\phi\in\IBr(G)}d_{\chi,\phi}\phi
=\sum_{\phi\in\IBr(G)} \left(\sum\limits_{\chi\in\Irr(G)} a_\chi d_{\chi,\phi}\right)\phi.
\]
By the linear independence of the irreducible Brauer characters we have $\sum\limits_{\chi\in\Irr(G)} a_\chi d_{\chi,\vhi}=1$.
Consequently, $\sum\limits_{\cO} |\cO|a_{\chi{_\cO}} d_{\chi{_\cO},\vhi}=1$ where $\cO$ runs through $A/G$-orbits in $\Irr(G)$, and for every orbit $\cO$ we choose $\chi_\cO\in\cO$.
As $A/G$ is an $\ell$-group, $|\cO|$ is either 1 or an $\ell$-number.
Therefore, there exists an $A$-invariant character $\chi\in\Irr(G)$ such that $\ell\nmid d_{\chi,\vhi}$. 
\end{proof}

Let $Y \unlhd X$ be finite groups and $\mathcal{Y} \subseteq \Irr(Y)\cup\IBr(Y)$. Following \cite[Definition~5.7]{CS17a}, we say that \textit{maximal extendibility holds for $\mathcal{Y}$ with respect to $Y \unlhd X$} if every element of $\mathcal{Y}$ extends its stabilizer in $X$. 

\begin{thm}\label{thm:ext-all}
Let $G\unlhd A$ be finite groups. Let $B$ be a $p$-block of $G$.
If maximal extendibility holds for $\Irr(B)$ with respect to $G\unlhd A$, then maximal extendibility holds for $\IBr(B)$ with respect to $G\unlhd A$.	
\end{thm}

\begin{proof}
Let $\vhi\in\IBr(B)$.
By \cite[Thm.~(8.29)]{Na98}, it suffices to prove that $\vhi$ extends to $X_\ell$ for every prime $\ell$ dividing $|A_\vhi/G|$.
Here $X_\ell/G$ is a Sylow $\ell$-subgroup of $A_\vhi/G$.
By a result of Green \cite[Thm.~(8.11)]{Na98}, $\vhi$ extends to $X_p$.

Fix a prime $\ell\ne p$.
By Lemma~\ref{lem:Gagola}, there exists an $X_\ell$-invariant character $\chi\in\Irr(G)$ such that $\ell\nmid d_{\chi,\vhi}$. 
As $d_{\chi,\vhi}\ne 0$, we have that $\chi\in\Irr(B)$.
By hypothesis, $\chi$ extends to $X_\ell$.
Hence Corollary~\ref{cor:ext-dec} implies that $\vhi$ extends to $X_\ell$.
This completes the proof.
\end{proof}

\begin{lem}\label{lem:gen-char}
Let $G\trianglelefteq A$ be finite groups.
\begin{enumerate}[\rm(1)]
\item Let $\nu\in\ZZ\Irr(G)$ be a generalized character of $G$ and write $\nu=\sum\limits_{\chi\in\Irr(G)}a_\chi\chi$.	
Assume that whenever $a_\chi\ne0$, the character $\chi$ extends to $A_\chi$.
If $\nu$ is $A$-invariant, then there exists a generalized character $\widetilde\nu\in\ZZ\Irr(A)$ of $A$ such that 
\begin{enumerate}[\rm(i)]
\item $\Res^{A}_{G}\widetilde\nu=\nu$, and	
\item if we write $\widetilde\nu=\sum\limits_{\widetilde\chi\in\Irr(A)}a_{\widetilde\chi}\widetilde\chi$, then $a_{\widetilde\chi}\ne 0$ occurs only when $a_\chi\ne 0$ for some $\chi\in\Irr(G\mid\widetilde\chi)$.
\end{enumerate}
\item Let $\vhi\in\IBr(G)$ be $A$-invariant.
Then $o_{A}(\vhi)$ divides the coefficient of $\vhi$ in the restriction to $G$ of any generalized Brauer character of $A$.
\end{enumerate}	
\end{lem}

\begin{proof}
(1)
Since $\nu$ is $A$-invariant, we have that $a_\chi=a_{\chi'}$ whenever $\chi$ and $\chi'$ are $A$-conjugate.
For each $A$-orbit in $\Irr(G)$ with non-zero coefficient, we choose an orbit representative $\chi$ and an extension $\widehat\chi$ to $A_{\chi}$.
As in the proof of \cite[Prop.~4.5(c)]{CS17a}, we set $\widetilde\nu=\sum\limits_\chi a_{\chi}\Ind_{A_{\chi}}^{A}(\widehat\chi)$,
where $\chi$ runs through a complete set of representatives of $A$-orbit in $\Irr(G)$ with non-zero coefficient.	
Then one checks that $\widetilde\nu$ satisfies the requirements.

(2) Let $\eta\in\IBr(A\mid\varphi)$.
Then by Clifford theory, $\Res^{A}_{G}\eta=e_{\eta}\vhi$ for some positive integer $e_{\eta}$.
According to \cite[Thm.~(8.16)]{Na98},
$o_{A}(\vhi)\mid e_{\eta}.$
On the other hand, if $\eta\notin\IBr(A\mid\varphi)$, then the coefficient of $\varphi$ in $\Res^{A}_{G}\eta$ is 0.
Therefore, we complete the proof.
\end{proof}

We generalize Lemma~\ref{lem:Schmid} as follows.

\begin{thm}\label{prop:extension}
Let $G\trianglelefteq A$ be finite groups, let $\mathcal B\subseteq\Irr(G)$ and $\mathcal C\subseteq\IBr(G)$ be $A$-stable sets of the same cardinality.
Suppose that $\Dec_{\cB,\cC}$ has nonzero determinant, and maximal extendibility holds for $\mathcal B$ with respect to $G\unlhd A$.
Then 
	\[o_{A_\vhi}(\vhi)\ \ \textrm{divides}\ \  \gcd( |A_\vhi/G|,\, |G|\vhi(1), \, \det\Dec_{\cB,\cC})\]
for every $\varphi\in\cC$.		
\end{thm}
	
\begin{proof}
Write $D=\Dec_{\cB,\cC}$.	
Thanks to \cite[Chap.~2, Thm.~7.3]{NT89} and \cite[Chap.~3, Thm.~5.7]{NT89}, it suffices to show that $o_{A_\vhi}(\vhi)$ divides $\det D$.

Let $U$ be the adjugate matrix of $D$.
Thus $U=(u_{\vhi,\chi})_{\vhi\in\cC,\,\chi\in\cB}$ is an integer matrix which has rows indexed by $\mathcal C$, columns indexed by $\mathcal B$.
Moreover, \[\sum\limits_{\chi\in\cB}u_{\varphi,\chi}d_{\chi,\psi}=\delta_{\varphi,\psi}\det D\]	
for every $\vhi,\psi\in\cC$, where $\delta_{\varphi,\psi}$ is the Kronecker symbol.

Recall that $d_{{}^h\chi,{}^h\vhi}=d_{\chi,\vhi}.$
Precisely, for $h\in A_\vhi$ if $P_{\mathcal B}(h)$ and $P_{\mathcal C}(h)$ are the relevant permutation matrices, then
$P_{\mathcal B}(h)D=DP_{\mathcal C}(h).$
Since $U=(\det D)D^{-1}$, this implies that $P_{\mathcal C}(h)U=UP_{\mathcal B}(h).$
Hence $u_{{}^h\vhi,{}^h\chi}=u_{\vhi,\chi}.$

Fix $\vhi\in\cC$. Then $u_{\varphi,\chi}$'s are constant when $\chi$ runs through an $A_\vhi$-orbit. 
Therefore, $\Xi_{\vhi}=\sum\limits_{\chi\in\cB}u_{\varphi,\chi}\chi$ is an $A_\vhi$-invariant generalized character of $G$.
By Lemma~\ref{lem:gen-char}(1), there exists a generalized character $\widetilde\Xi_\vhi$ of $A_\vhi$ with $\Res_{G}^{A_\varphi}(\widetilde\Xi_\vhi)=\Xi_\vhi$.
In particular, the coefficient of $\varphi$ in $\Res_{G}^{A_\varphi}({\widetilde\Xi_\vhi}^0)={\Xi_\vhi}^0$ is $\det D$.
Therefore, $o_{A_\vhi}(\vhi)$ divides $\det D$ by Lemma~\ref{lem:gen-char}(2).	
\end{proof}	
	
Similar as Corollary~\ref{cor:ext-dec}, we have:

\begin{cor}	\label{cor:cor:ext-det}
Let $G\trianglelefteq A$ be finite groups, let $\mathcal B\subseteq\Irr(G)$ and $\mathcal C\subseteq\IBr(G)$ be $A$-stable sets of the same cardinality.
Suppose that $\Dec_{\cB,\cC}$ has nonzero determinant, and maximal extendibility holds for $\mathcal B$ with respect to $G\unlhd A$. 
If $\gcd( |A/G|_{p'},\, \det\Dec_{\cB,\cC})=1$, then maximal extendibility holds for $\mathcal C$ with respect to $G\unlhd A$. 
\end{cor}	
	
The following result follows from Corollary~\ref{cor:cor:ext-det} immediately.

\begin{cor}\label{cor:extension}
	Let $G\trianglelefteq A$ be finite groups, let $\mathcal B\subseteq\Irr(G)$ and
	$\mathcal C\subseteq\IBr(G)$ be $A$-stable sets of the same
	cardinality.
	Suppose that $\Dec_{\cB,\cC}$ is unimodular. If maximal extendibility holds for $\mathcal B$ with respect to $G\unlhd A$, then maximal extendibility holds for $\mathcal C$ with respect to $G\unlhd A$.
\end{cor}

We conclude this section with an application of Theorem~\ref{thm:ext-all}.

\begin{rmk}
Let $\mathbf G$ be a simple simply-connected algebraic group and $F:\mathbf G\to\mathbf G$ a Frobenius endomorphism. As in \cite[\S2.B]{MS16}, we denote by $D$ the group generated by field and graph automorphisms.
\begin{enumerate}[\rm(1)]
\item Suppose that $\bG$ is of type $\mathsf E_6$.
By \cite[Thm.~3.10]{CS19} and its proof, 
maximal extendibility holds for $\Irr(\bG^F)$ with respect to $\bG^F\unlhd \bG^F\rtimes D$.
Therefore, Theorem~\ref{thm:ext-all} implies that maximal extendibility holds for $\IBr(\bG^F)$ with respect to $\bG^F\unlhd \bG^F\rtimes D$ for any prime $p$.
\item Suppose that $\mathbf G$ is not of type $\mathsf A$ and $p$ is bad for $\mathbf G$. Then maximal extendibility holds for $\IBr(\bG^F)$ with respect to $\bG^F\unlhd \bG^F\rtimes D$.
In fact, it suffices to consider the case that $D$ is non-cyclic.
Therefore, $\mathbf G^F$ is of type $\mathsf D_n$ or $\mathsf E_6$.
By (i), we may assume that $\mathbf G^F$ is of type $\mathsf D_n$ and $p=2$.
Let $\varphi\in\IBr_2(\bG^F)$.
By \cite[Thm.~(8.29)]{Na98}, it remains to prove that $\varphi$ extends to $\bG^F\rtimes X_\ell$ for every prime $\ell$, where $X_\ell$ is a Sylow $\ell$-subgroup of $D_\varphi$.
According to \cite[Thm.~(8.11)]{Na98}, $\vhi$ extends to $\bG^F\rtimes X_2$.
By \cite[Lemma~6.2]{Sp25} and Theorem~\ref{thm:ext-all}, $\vhi$ extends to $\bG^F\rtimes X_\ell$ for any odd prime $\ell$, as desired.
\end{enumerate}
\end{rmk}


\section{Stabilizers and extendibility}\label{sec:A(infty)}

The following is the main theorem of this section.

\begin{thm}\label{thm:ord-Bra}	
Let $A$ be a finite group, and let $G,C\unlhd A$ and $D\le A$ be such that $A=CD$, $G=C\cap D$, and $C/G$ is abelian.
Suppose further that $C/G\Z(C)$ is either cyclic or an elementary abelian $\ell$-group of rank~$2$, where $\ell$ is any prime.
Assume that for every $\chi\in\Irr(G)$, there exists a $C$-conjugate $\chi_0$ of $\chi$ such that $A_{\chi_0}=C_{\chi_0}D_{\chi_0}$ and $\chi_0$ extends to $D_{\chi_0}$.
Then for every $\psi\in\IBr(G)$, there exists a $C$-conjugate $\psi_0$ of $\psi$ such that $A_{\psi_0}=C_{\psi_0}D_{\psi_0}$ and $\psi_0$ extends to $D_{\psi_0}$.
\end{thm}

Let $G$ be a finite group, let $H$ be a subgroup and let $M$ ba a $\mathbb{Z}G$-module. 
We will make use of the \emph{restriction} $\res_H^G: \textup{H}^n(G, M)\rightarrow \textup{H}^n(H, M)$ and \emph{corestriction} $\cores_H^G: \textup{H}^n(H, M)\rightarrow \textup{H}^n(G, M)$ between cohomology groups; see for example \cite[Chap.~III, \S9]{Br82} for details.
We also write $\textup{tr}_H^G: M^H\to M^G$,
$x\mapsto \sum_{t\in [G/ H]} t\cdot x$, where $[G/ H]$ denotes a complete set of representatives of $H$-cosets in $G$.
When $n= 0$, we have $\cores_H^G=\textup{tr}_H^G$ (see \cite[Chap.~III, \S9 (B)]{Br82}).

We will use the notation $c_A(\psi)$ defined in page~\pageref{notation-cohomology}.

\begin{lem}\label{lem:coho-cal}
Let $A$ be a finite group, and let $H,C\le A$. Let $G=C\cap H$, and suppose that $G\unlhd C,H$ and $H\le\N_A(C)$.
Let $\psi\in\IBr(G)$ be $H$-invariant. Assume that $C/G$ is abelian and $\psi$ extends to $C_\psi$.
Let $Q=C/C_\psi$, and let $\Omega$ denote the set of $C$-conjugates of $\psi$.
For every $\eta\in\Omega$, let $q_\eta$ be the unique element of $Q$ such that $\eta=\psi^{q_\eta}$.
Let $E= H/ G$. Write $\Omega=\coprod\limits_{i=1}^s\Omega_i$ as a union of $E$-orbits.
For each $i$, fix $\eta_i\in\Omega_i$.
Set $U_i=H_{\eta_i}/G$, $q_i=\prod\limits_{\eta\in\Omega_i}q_\eta$,
$\psi_i=\psi^{q_i}$ and $d_i=|H:H_{\eta_i}|$.
Then $\psi_i$ is $H$-invariant, and 
\begin{equation}\label{equ:cor-eq1}
\cores^{E}_{U_i}(c_{H_{\eta_i}}(\eta_i))=c_H(\psi)^{d_i-1}c_H(\psi_i).
\addtocounter{thm}{1}\tag{\thethm}
\end{equation}

Let $m_1,\ldots,m_s$ be integers, $d=\sum\limits_{i=1}^s m_id_i$, $q=\prod\limits_{i=1}^s{q_i}^{m_i}$ and $\psi'=\psi^{q}$.
Then $\psi'$ is $H$-invariant and 
\begin{equation}\label{equ:cor-eq2}
 c_H(\psi)^{d-1}c_H(\psi')=\prod_{i=1}^s\cores^{E}_{U_i}(c_{H_{\eta_i}}(\eta_i))^{m_i}.
 \addtocounter{thm}{1}\tag{\thethm}
\end{equation}
\end{lem}

\begin{proof}
{\it Step 1}.
Let $\rho$ be a representation of $G$ affording $\psi$. 
Let $\widetilde\rho$ be an extension of $\rho$ to $C_\psi$, and $\mathcal P$ be a projective representation of $H$  that extends $\rho$ as in page~\pageref{def:proj}; in particular, $\mathcal P(hg)=\mathcal P(h)\rho(g)$ and $\mathcal P(gh)=\rho(g)\mathcal P(h)$ for every $h\in H$ and $g\in G$.
For every $h\in H$ and $c\in C_\psi$, the matrix $\mathcal P(h)\widetilde\rho(c)\mathcal P(h)^{-1}\widetilde\rho(hch^{-1})^{-1}$ commutes with $\rho(g)$ for every $g\in G$:
\begin{equation*}
\begin{aligned}
\rho(g)\mathcal P(h)\widetilde\rho(c)\mathcal P(h)^{-1}\widetilde\rho(hch^{-1})^{-1}&= \mathcal{P}(h)\rho(h^{-1}gh)\widetilde\rho(c)\mathcal P(h)^{-1}\widetilde\rho(hch^{-1})^{-1}\\
&= \mathcal P(h)\widetilde\rho(c)\rho(c^{-1}h^{-1}g^{-1}hc)^{-1}\mathcal P(h)^{-1}\widetilde\rho(hch^{-1})^{-1}\\
&= \mathcal{P}(h)\widetilde{\rho}(c)\mathcal{P}(h)^{-1}\widetilde\rho(hch^{-1})^{-1}\rho(g).
\end{aligned}
\end{equation*}
By Schur's lemma, there is a scalar $\lambda_{h}:C_\psi\to k^\times$ such that $\mathcal P(h)\widetilde\rho(c)\mathcal P(h)^{-1}=\lambda_h(c)\widetilde\rho(hch^{-1})$.

It can be checked that $\lambda_h$ is a group homomorphism, $G\le\ker(\lambda_h)$:
\begin{equation*}
\begin{aligned}
\lambda_h(c_1c_2)&= \mathcal{P}(h)\widetilde\rho(c_1c_2)\mathcal{P}(h)^{-1}\widetilde\rho(hc_1c_2h^{-1})^{-1}\\
&= \lambda_h(c_1)\widetilde\rho(hc_1h^{-1})\lambda_h(c_2)\widetilde\rho(hc_1h^{-1})^{-1}\\
&= \lambda_h(c_1)\lambda_h(c_2).
\end{aligned}
\end{equation*} For later use, we give two formulas about $\lambda_h$ here:

(1) $\lambda_{hg}(c)= \lambda_{h}(gcg^{-1})= \lambda_h(c)$ for every $h\in H$ and $g\in G$.

(2) Let $\alpha$ be the 2-cocycle induced by $\mathcal{P}$. Then for any $h, k\in H$, 
\begin{equation*}
\begin{aligned}
\lambda_{hk}(c)\lambda_k(c)^{-1}
&= \alpha(h, k)\mathcal{P}(h)\mathcal{P}(k)\widetilde{\rho}(c)\alpha(h, k)^{-1}\mathcal{P}(k)^{-1}\mathcal{P}(h)^{-1}\widetilde{\rho}(hkck^{-1}h^{-1})^{-1}\lambda_k(c)^{-1}\\
&= \mathcal{P}(h)\lambda_k(c)^{-1}\mathcal{P}(k)\widetilde{\rho}(c)\mathcal{P}(k)^{-1}\mathcal{P}(h)^{-1}\widetilde{\rho}(hkck^{-1}h^{-1})^{-1}\\
&= \lambda_h(kck^{-1}).
\end{aligned}
\end{equation*}

{\it Step 2}.
Note that $E=H/G$ acts on abelian groups $C/G$, $C/C_\psi$ and $C_\psi/G$. Let $U$ be a subgroup of $E$.
Consider the short exact sequence \[1\to C_\psi/G\to C/G\to C/C_\psi\to 1\] of $U$-modules, and then we have a homomorphism
$d_U^0:(C/C_\psi)^U=\textup{H}^0(U,C/C_\psi)\to\textup{H}^1(U,C_\psi/G)$ (see, e.g., \cite[Chap. III, Prop.~6.1]{Br82}).
Let $W=k^\times \times(C_\psi/G)$ and define the action of $E=H/G$ on $W$ via $\overline h\cdot(z,\overline c)=(z\lambda_{h}(c),\overline{hch^{-1}})$.
From the short exact sequence \[1\to k^\times\to W\to C_\psi/G\to 1\] of $U$-modules, we get a homomorphism $d_U^1:\textup{H}^1(U,C_\psi/G)\to\textup{H}^2(U,k^\times)$.
Consider the composite homomorphism \[\tau_U=d_U^1\circ d_U^0:(C/C_\psi)^U\to \textup{H}^2(U,k^\times).\]
By \cite[Chap. III, \S9 (B)]{Br82}, 
\begin{equation}\label{equ:cor1}
\cores^E_U(\tau_U(v))=\tau_E(\textup{tr}^E_U(v)).
\addtocounter{thm}{1}\tag{\thethm}
\end{equation}

{\it Step 3}.
Calculate $\tau_U$. 

Let $H_U$ be the inverse image of $U$ in $H$. Let $v\in (C/C_\psi)^U$. By abusing of notations, we let $\delta$ be the natural map between cochains:
\[
\begin{array}{ccccccccc}
1 & \longrightarrow &
C^0(U,C_\psi/G) & \longrightarrow &
C^0(U,C/G) & \longrightarrow &
C^0(U,C/C_\psi) & \longrightarrow & 1
\\[4pt]
&& \Big\downarrow
&& \Big\downarrow\,{\scriptstyle\delta}
&& \Big\downarrow
&&
\\[4pt]
1 & \longrightarrow &
C^1(U,C_\psi/G) & \longrightarrow &
C^1(U,C/G) & \longrightarrow &
C^1(U,C/C_\psi) & \longrightarrow & 1.
\end{array}
\]
Let $\widetilde{v}$ be a representative of $v$ in $\textup{Z}^0(U, C/ C_\psi)$, and let $w$ be a preimage of $\widetilde{v}$ in $\textup{C}^0(U, C/ G)$. Note that $\delta(w)$ takes values in $C_\psi/ G$. Then $\delta(w)\in \Z^1(U, C_\psi/ G)$ and $[\delta(w)]$ is $d_U^0(v)$. Repeat this procedure, we can get $\tau_U(v)$: Choose $t\in C$ with $tC_\psi=v$.
Define a 2-cocycle $\kappa_t$ of $U$ by letting $\kappa_t(xG,yG):=\lambda_x(t^{-1}yty^{-1})$ for $x,y\in H_U$. We have $[\kappa_t]= \tau_U(v)$.

{\it Step 4}. Calculate $c_{H_U}(\psi^v)$.

Similar as in \cite[Lemma~2.7]{FR21} and its proof, we define $S_t(x)=\widetilde\rho(t^{-1}xtx^{-1})\mathcal P(x)$ for $x\in H_U$.
For any $x, y\in H_U$, 
\begin{equation*}
\begin{aligned}
S_t(xy)S_t(y)^{-1}S_t(x)^{-1}
=&\widetilde{\rho}(t^{-1}xyty^{-1}x^{-1})\alpha(x, y)\mathcal{P}(x)\widetilde{\rho}(t^{-1}yty^{-1})^{-1}\mathcal{P}(x)^{-1}\widetilde{\rho}(t^{-1}xtx^{-1})^{-1}\\
=& \alpha(x, y)\lambda_x(t^{-1}yty^{-1})^{-1}.
\end{aligned}
\end{equation*}
So $S_t: H_U\rightarrow \GL_{\psi(1)}(k)$ is a projective representation with factor set $\alpha\mid_{U\times U}\cdot \kappa_t^{-1}$. For any $g\in G$, $S_t(g)= \widetilde{\rho}(t^{-1}gtg^{-1})\mathcal{P}(g)= \rho(t^{-1}gt)$. Also note that 
\begin{equation*}
S_t(xg)\rho(t^{-1}gt)^{-1}= \widetilde{\rho}(t^{-1}xtx^{-1})\rho(xt^{-1}gtg^{-1}x^{-1})\mathcal{P}(x)\rho(t^{-1}gtg^{-1})^{-1}= S_t(x). 
\end{equation*}
Therefore, $S_t(xg)= S_t(x){\rho^v(g)}$. Similarly, we get $S_t(gx)= {\rho^v}(g)S_t(x)$. Thus $S_t$ is a projective representation associated to $\rho^v$.
Consequently, 
\begin{equation}\label{equ:cor2}
	c_{H_U}(\psi^v)=\res^E_U(c_H(\psi))\tau_U(v)^{-1}.
	\addtocounter{thm}{1}\tag{\thethm}
\end{equation}

{\it Step 5}.
Note that $q_i=\textup{tr}_{U_i}^{E}(q_{\eta_i})$ and $\eta_i=\psi^{q_{\eta_i}}$.
So 
\begin{align*}
\cores^E_{U_i}(c_{H_{\eta_i}}(\eta_i))&=\cores^E_{U_i}\circ\res^E_{U_i}(c_H(\psi))\cdot  \cores^E_{U_i}(\tau_{U_i}(q_{\eta_i})^{-1})\\
&=c_H(\psi)^{d_i}\tau_E(q_i)^{-1},
\end{align*}
where the first equality follows by (\ref{equ:cor2}), and the last equality follows by \cite[Chap. III, Prop. 9.5]{Br82} and (\ref{equ:cor1}).
Since $\psi_i=\psi^{q_i}$, by (\ref{equ:cor2}) again, 
$c_H(\psi_i)=c_H(\psi)\tau_E(q_i)^{-1}$.
Hence (\ref{equ:cor-eq1}) holds.

Since $\tau_E$ is a homomorphism, we have 
\begin{align*}
&\prod\limits_{i=1}^s\cores^E_{U_i}(c_{H_{\eta_i}}(\eta_i))^{m_i}=\prod\limits_{i=1}^s(c_H(\psi)^{d_i}\tau_E(q_i)^{-1})^{m_i}
=c_H(\psi)^{d}\tau_E(\prod\limits_{i=1}^s{q_i}^{m_i})^{-1}\\
=&c_H(\psi)^{d}\tau_E(q)^{-1}=c_H(\psi)^{d-1}c_H(\psi'),
\end{align*}
which gives (\ref{equ:cor-eq2}). 
\end{proof}

\begin{prop}\label{prop:ord-Bra}
	Let $A$ be a finite group, $G,C\unlhd A$ and $D\le A$ satisfy that $A=CD$, $G=C\cap D$ and $C/G$ is abelian.
	Let $\Omega$ be a $C$-orbit in $\IBr(G)$.
	Suppose that there is a generalized Brauer character $\Theta$ of $H:=D_\Omega$ such that every irreducible constituent with non-zero coefficient in $\Theta$ lies over $\Omega$ and $\Res^{H}_{G}\Theta=\sum\limits_{\eta\in\Omega}n_\eta\eta$, $\sum\limits_{\eta\in\Omega}n_\eta=1$ and $n_{\eta^h}=n_\eta$ for every $h\in H$.
	Assume that whenever $|\Omega|>1$, maximal extendibility holds for $\Omega$ with respect to $G\unlhd C$.
	Then there exists $\psi_0\in\Omega$ such that $A_{\psi_0}=C_{\psi_0}D_{\psi_0}$ and $\psi_0$ extends to $D_{\psi_0}$.
\end{prop}

\begin{proof}
Let $K=C_\psi$ for some $\psi\in\Omega$, then $K=C_\eta$ for every $\eta\in\Omega$ since $C/ G$ is abelian. Therefore, $H$ normalizes $K$. So $H$ acts on $C/G$ and $M:=K/G$.

We first claim that there exists an $H$-invariant element $\psi_0\in\Omega$ and $c_\eta\in C/G$ (for each $\eta\in\Omega$) such that \[\eta={\psi_0}^{c_\eta}\quad  \text{and}\quad \prod\limits_{\eta\in\Omega}{c_\eta}^{n_\eta}=1\in C/G.\]
Fix $\psi\in\Omega$.
For each $\eta\in\Omega$, we choose $t_\eta\in C/G$ such that $\eta=\psi^{t_\eta}$, and set $t=\prod\limits_{\eta\in\Omega}{t_\eta}^{n_\eta}$.
For $h\in H$, let $a_h\in C/G$ with $\psi^h=\psi^{a_h}$.
Then $\eta^h=(\psi^{t_\eta})^h=(\psi^h)^{{t_\eta}^h}=\psi^{a_h{t_\eta}^h}$.
It follows from $\eta^h=\psi^{t_{\eta^h}}$ that $a_h{t_\eta}^hM=t_{\eta^h}M$.
Therefore, $(a_h{t_\eta}^h)^{n_\eta}M=(t_{\eta^h})^{n_\eta}M$ and 
\[
a_h t^h M=\prod_{\eta\in\Omega}(a_h{t_\eta}^h)^{n_\eta}M=\prod_{\eta\in\Omega}(t_{\eta^h})^{n_\eta}M=tM.
\]
Let $\psi_0=\psi^t$.
Then ${\psi_0}^h=(\psi^t)^h=(\psi^h)^{t^h}=\psi^{a_ht^h}=\psi^t=\psi_0$, whence $\psi_0$ is $H$-invariant.
Finally, let $c_\eta=t^{-1}t_\eta$ for each $\eta\in\Omega$.
Then $\eta={\psi_0}^{c_\eta}$, and it can be checked that $\prod\limits_{\eta\in\Omega}{c_\eta}^{n_\eta}=1\in C/G$.
Hence the claim holds.

If $|\Omega|=1$, then we write $\Omega=\{\psi_0\}$, and we have $\Res^{H}_{G}\Theta=\psi_0$.
Hence $\psi_0$ is $H$-invariant, and by Lemma~\ref{lem:gen-char}(2), $\psi_0$ extends to $H$.
Therefore, we assume that $|\Omega|>1$, and maximal extendibility holds for $\Omega$ with respect to $G\unlhd C$.

Write $\Omega=\coprod\limits_{i=1}^s\Omega_i$ as a union of $H$-orbits.
Fix $\eta_i\in\Omega_i$. Let $d_i=|H:H_{\eta_i}|$ and $m_i=n_{\eta_i}$.
Thus $\sum\limits_{i=1}^sd_im_i=\sum\limits_{\eta\in\Omega}n_\eta=1$.
Let $q_\eta=c_\eta M$ so that $\prod\limits_{i=1}^s(\prod\limits_{\eta\in\Omega_i}q_\eta)^{m_i}=1$.
Set $E=H/G$ and $U_i=H_{\eta_i}/G$.
Then by Lemma~\ref{lem:coho-cal},
\[c_H(\psi_0)=\prod\limits_{i=1}^s \cores_{U_i}^E(c_{H_{\eta_i}}(\eta_i))^{m_i}.\]
The coefficient of $\eta_i$ in the generalized Brauer character $\Res_{G}^{H_{\eta_i}}(\Res^H_{H_{\eta_i}}\Theta)$
is $m_i$, and thus Lemma~\ref{lem:gen-char}(2) implies that $c_{H_{\eta_i}}(\eta_i)^{m_i}=1$.
Since corestrictions are group homomorphisms, we have $c_H(\psi_0)=1$, and thus $\psi_0$ extends to $H$. If $cd\in CD$ fixes $\psi_0$, then for any $c'\in C$ we have $d(c'\psi_0)= dc'd^{-1}(d\psi_0)= dc'd^{-1}(c^{-1}\psi_0)\in \Omega$. Thus $d\in H\leq D_{\psi_0}$ and $c\in C_{\psi_0}$.
This completes the proof.  
\end{proof}

\begin{lem}\label{lem:exten-central-pro}
Let $G$ be a normal subgroup of a finite group $A$ and $Z\le\Z(A)$.
Let $\chi\in\Irr(G)\cup\IBr(G)$ be $A$-invariant.
Then there is an $A$-invariant extension of $\chi$ to $GZ$.

In particular, if $A/GZ$ is cyclic, then $\chi$ extends to $A$.
\end{lem}

\begin{proof}
Let $\tau$ be the irreducible (Brauer) character of $G\cap Z$ lying down $\chi$.	
Then there is an extension $\lambda$ of $\tau$ to $Z$.	
So $\lambda$ is $A$-invariant, since $Z\le\Z(A)$.
Then $\chi\cdot\lambda$ (under the notation of \cite[Thm.~10.7]{Na18} and \cite[Lemma~2.5]{FS23}) is an $A$-invariant extension of $\chi$ to $GZ$.	
\end{proof}

Let $G$ be a finite group acting on a finite set $X$. A \emph{$G$-transversal} is a subset $S\subseteq X$ such that the size of the intersection of $S$ and every $G$-orbit is exactly 1. Now we are able to establish Theorem~\ref{thm:ord-Bra}.

\begin{proof}[Proof of Theorem~\ref{thm:ord-Bra}]
By the assumption, every proper subgroup of $C/G\Z(C)$ is cyclic.
Then Lemma~\ref{lem:exten-central-pro} implies that every Brauer character $\psi\in\IBr(G)$ extends to $C_\psi$ if $C_\psi<C$ (In Lemma~\ref{lem:exten-central-pro}, we let $A= C_\psi$, $Z= \Z(C)\cap C_\psi$). 	
According to Proposition~\ref{prop:ord-Bra}, it suffices to show that
for every $C$-orbit $\Omega$ in $\IBr(G)$, there is a generalized Brauer character $\Theta$ of $H:=D_\Omega$ such that every irreducible constituent with non-zero coefficient in $\Theta$ lies over $\Omega$ and $\Res^{H}_{G}\Theta=\sum\limits_{\eta\in\Omega}n_\eta\eta$, $\sum\limits_{\eta\in\Omega}n_\eta=1$ and $n_{\eta^h}=n_\eta$ for every $h\in H$.	
	
By assumption, there is a $D$-stable $C$-transversal $\cT$ in $\Irr(G)$ such that maximal extendibility holds for $\cT$ with respect to $G\unlhd D$: Since $D$ acts on $C$-orbits in $\Irr(G)$, we let $X$ be the set of $D$-orbits. Then let $Y$ be a $D$-orbit and choose $\mathcal{O}\in Y$. Choose a character $\chi_0\in \mathcal{O}$ satisfying $A_{\chi_0}= C_{\chi_0}D_{\chi_0}$ and $\chi_0$ extends to $D_{\chi_0}$. If $d_1\chi_0$, $d_2\chi_0$ lie in the same $C$-orbit, there exists $c\in C$ such that $cd_1\chi_0= d_2\chi_0$. Thus there exists $c'\in C_{\chi_0}$, $d'\in D_{\chi_0}$ such that $d_2^{-1}cd_1= c'd'$. So $d'd_1^{-1}d_2\in G$ and $d_1^{-1}d_2\in D_{\chi_0}$. Therefore, $d_1\chi_0= d_2\chi_0$. Now we get a $D$-stable $C$-transversal $\bigcup\limits_{Y\in X}D\chi_0$.

Let $P\colon \ZZ\Irr(G)\to\ZZ\cT$ be the $\ZZ$-linear map that sends each irreducible character in $\Irr(G)$ to its $C$-conjugate in $\cT$.
Then $P(c\nu)=P(\nu)$ and $P(d\nu)=dP(\nu)$ for every $c\in C$, $d\in D$ and $\nu\in\ZZ\Irr(G)$.

Now let $\Omega$ be a $C$-orbit in $\IBr(G)$. Fix $\psi\in\Omega$.
Let $\widehat\psi$ be the generalized character of $G$ with $\widehat\psi(g)=\psi(g_{p'})$, as in the proof of Lemma~\ref{lem:Gagola}.
Let $\omega=P(\widehat\psi)$.
It can be checked that $\omega$ is $H$-invariant.
Lemma~\ref{lem:gen-char}(1) implies a generalized character $\widetilde\omega$ of $H$ with $\Res^H_G\widetilde\omega=\omega$.

For every generalized Brauer character $\mu$ of $G$, we denote by $s_\Omega(\mu)$ the sum of its coefficients on $\Omega$. 
There is a decomposition $\widehat{\psi}= \sum_{\chi\in \Irr(G)}a_\chi\chi$, then $s_\Omega(\widehat{\psi}^0)= \sum a_\chi s_\Omega(\chi^0)$. Also, $s_\Omega(\omega^0)= s_\Omega(\sum a_\chi P(\chi)^0)= \sum a_\chi s_\Omega(P(\chi)^0)$. For every $\chi\in \Irr(G)$, there exists $c_\chi\in C$ such that $P(\chi)= c_\chi\cdot \chi$. Thus $s_\Omega(P(\chi)^0)= s_\Omega(c_\chi\cdot \chi^0)= s_\Omega(\chi^0)$.
Therefore, $s_\Omega(\omega^0)=s_\Omega(\widehat\psi^0)=s_\Omega(\psi)=1$ since $\psi= \widehat{\psi}^0$.

Write $\widetilde\omega^0=\sum\limits_{\widetilde\eta\in\IBr(H)}a_{\widetilde\eta}\widetilde\eta$.
Define $\Theta=\sum\limits_{\widetilde\eta}a_{\widetilde\eta}\widetilde\eta$, where $\widetilde\eta$ runs through the irreducible Brauer characters of $H$ lying over some Brauer character in $\Omega$.
For any Brauer character in $\Omega$, its coefficient in $\Res^H_G\Theta$ is equal to that in $\Res^H_G(\widetilde\omega^0)=\omega^0$.
Therefore,
$\Res^{H}_{G}\Theta=\sum\limits_{\eta\in\Omega}n_\eta\eta$, $\sum\limits_{\eta\in\Omega}n_\eta=1$ and $n_{\eta^h}=n_\eta$ for every $h\in H$.
This completes the proof.
\end{proof}


\section{Proofs of Main Theorems}\label{sec:size-series}

Let $q$ be a power of a prime $r$. We let $\overline\FF_q$ be an algebraic closure of the finite field $\FF_q$ of $q$ elements. 
Suppose that $\bG$ is a connected reductive algebraic group over $\overline\FF_q$ and $F:\bG\to\bG$ is a Frobenius endomorphism endowing $\bG$ with an $\FF_q$-structure. Let $\bG^*$ be Langlands dual to $\bG$ with corresponding Frobenius endomorphism also denoted $F$. 
We consider the modular representations with respect to the prime $p$.

We recall the following result of Geck and Hiss; see also \cite[Thm.~14.4]{CE04}.

\begin{thm}[Geck--Hiss \cite{GH91}, Geck \cite{Ge93}]\label{thm:basic-set}
	Assume $p$ is good for $\bG$, and does not divide  $q|(\Z(\bG)/\Z^\circ(\bG))^F|$. Then $\cE(\bG^F,p')$ is a basic set for $\bG^F$.
\end{thm}

In what follows, we assume that that $\bG$ is simple and of simply connected type.
Let $\bG\hookrightarrow\tbG$ be a regular embedding with Frobenius endomorphism also denoted as $F$.
Write $\pi:\tbG^*\to\bG^*$ the corresponding epimorphism.
The group $\tbG^F\rtimes D$ is well-defined and induces all automorphisms of $\bG^F$, see~\cite[Thm.~2.5.1]{GLS98}. 
For a explicit description of $\tbG$ and $D$, we refer to \cite[\S2.B]{MS16}.
The following is called the \emph{$\mathrm{A}(\!\infty\!)$ condition} by Cabanes and Sp\"ath \cite[Definition 2.2]{CS19} and played an instrumental role in the proof of the McKay conjecture.

\begin{thm}[Cabanes--Sp\"ath {\cite{CS17a,CS17b,CS19}}, Sp\"ath {\cite{Sp23,Sp25}}]
	\label{ord-Ainfty}
	Every $\chi\in \Irr(\bG^F)$ has a $\tbG^F$-conjugate $\chi_0$ such that $(\tbG^FD)_{\chi_0}=\tbG^F_{\chi_0}D_{\chi_0}$ and $\chi_0$ extends to $\bG^FD_{\chi_0}$.
\end{thm}

We restate the Brauer $A(\infty)$ condition as follows.

\begin{defn}[Brauer $A(\infty)$ condition]\label{defn:BrauerAinf}
	Every $\varphi\in \IBr(\mathbf{{G}}^F)$ has a $\mathbf{\widetilde{G}}^F$-conjugate $\varphi_0$ such that $(\mathbf{\widetilde{G}}^FD)_{\varphi_0}= \mathbf{\widetilde{G}}^F_{\varphi_0}D_{\varphi_0}$ and $\varphi_0$ extends to $\mathbf{{G}}^FD_{\varphi_0}$.
\end{defn}

Thus Theorem~\ref{mainthm-1} follows from Theorem~\ref{thm:ord-Bra} and Theorem~\ref{ord-Ainfty} immediately.

\begin{lem}\label{lem:E6-series}
Suppose that $(\bG,F)$ is of split type $\ty E_6$, and $s\in{\bG^*}^F$ is a semisimple element such that the group $\C_{\bG^*}(s)^F/\C_{\bG^*}^\circ(s)^F$ is nontrivial. Let $\tilde s$ be a semisimple element of ${\tbG}^{*F}$ such that $\pi(\tilde s)=s$.
Then the following holds.
\begin{enumerate}[\rm(1)]
\item $|\cE(\bG^F,s)|\le23$ or $|\cE(\bG^F,s)^{\tbG^F}|\le 5$.
\item Let $L$ be the set stabilizer of $\cE(\tbG^F,\tilde s)$ in $\Lin(\tbG^F/\bG^F)$. Then $|\cE(\tbG^F,\tilde s)|\le23$ or $|\cE(\tbG^F,\tilde s)^{L}|\le 5$.
\end{enumerate}
\end{lem}

\begin{proof}
The group $\C_{\bG^*}(s)^F/\C_{\bG^*}^\circ(s)^F$ is of order 3 and the group $\C_{\bG^*}(s)$ is disconnected, 
since $\C_{\bG^*}(s)^F/\C_{\bG^*}^\circ(s)^F$ is nontrivial.
Let $\mathcal U_s$ be the set of unipotent characters of $\C_{\bG^*}^\circ(s)^F$.
Let $k=|{\mathcal U_s}^{\C_{\bG^*}(s)^F}|$, and $l$ be the number of $\C_{\bG^*}(s)^F$-orbits on $\mathcal U_s\setminus {\mathcal U_s}^{\C_{\bG^*}(s)^F}$.
Hence $|\mathcal U_s|=k+3l$.
By Lusztig's Jordan decomposition for disconnected groups (cf. \cite{Lu88}, see also \cite[Thm.~2.6.21]{GM20}) and Clifford theory, we have that $|\cE(\bG^F,s)|=3k+l$, $|\cE(\bG^F,s)^{\tbG^F}|=l$, $|\cE(\tbG^F,\tilde s)|=k+3l$ and $|\cE(\tbG^F,\tilde s)^{L}|=k$.

Since $\C_{\bG^*}(s)$ is disconnected, by \cite[\S~3.B]{Bo05} and \cite[Thm.~2.2(3)]{Br12}, the group $\C_{\bG^*}^\circ(s)$ has one of the types
1, $\ty A_1$, $\ty A_1^3$, $\ty A_1^4$, $\ty D_4$, $\ty A_2^3$.
The upper bound number of unipotent characters are as in the following table.
\begin{table}[htbp]
\[\begin{array}{c|rrrrrr}
\text{type}&1&\ty A_1&\ty A_1^3&\ty A_1^4&\ty D_4&\ty A_2^3\\
\hline
|\mathcal U_s|&1&2&8&16&14&27.
\end{array}\]
\end{table}
Recall that the numbers for $\ty A_1$ and $\ty A_2$ are the numbers of partitions of 2 and 3, and the group of type $\ty D_4$ has 14 unipotent characters (see, e.g., \cite[Table~2]{Ma20}).

If $\C_{\bG^*}(s)$ is not of type $\ty A_2^3$, then $|\mathcal U_s|=k+3l\le 16$ and thus $|\cE(\tbG^F,\tilde s)|\le 16$.
Moreover, $l\le 5$, whence $|\cE(\bG^F,s)^{\tbG^F}|\le 5$.

Therefore, we assume that $\C_{\bG^*}(s)$ is of type $\ty A_2^3$.
The group $\C_{\bG^*}(s)^F/\C_{\bG^*}^\circ(s)^F$ cyclically permutes the three 
$\ty A_2$ factors.  Since $\C_{\bG^*}(s)^F/\C_{\bG^*}^\circ(s)^F$ is fixed by $F$, the permutation of the factors induced by $F$ commutes with this three cycle.  It is therefore either the identity or a three cycle.  If $F$ fixes the factors, then
$|\mathcal U_s|=27$, $k=3$ and $l=8$.
So,
$|\cE(\bG^F,s)|=17$.
If $F$ cyclically permutes the factors, then the rational group has type $\ty A_2(q^3)$, and $|\mathcal U_s|=3$, $k=3$ and $l=0$. Therefore,
$|\cE(\bG^F,s)|=9$ and $|\cE(\tbG^F,\tilde s)|=3$.
\end{proof}

\begin{lem}\label{lem:D4-series}
	Suppose that $(\bG,F)$ is of split type $\ty D_4$, and $s\in{\bG^*}^F$ is a semisimple element. Let $\tilde s$ be a semisimple element of ${\tbG}^{*F}$ such that $\pi(\tilde s)=s$.
	Then $|\cE(\bG^F,s)|\le16$ and $|\cE(\tbG^F,\tilde s)|\le16$.
\end{lem}
	
\begin{proof}
Let $\mathcal U_s$ be the set of unipotent characters of $\C_{\bG^*}^\circ(s)^F$.

Suppose first that $\C_{\bG^*}(s)^F/\C_{\bG^*}^\circ(s)^F$ is trivial.  
Then $|\cE(\bG^F,s)|=|\cE(\tbG^F,\tilde s)|=|\mathcal U_s|$.
The possible types of $\C_{\bG^*}^\circ(s)$ are 1, $\ty A_1$, $\ty A_1^2$, $\ty A_1^3$, $\ty A_1^4$, $\ty A_2$, $\ty A_3$, $\ty D_4$.
The maximal numbers of unipotent characters are as in the following table.
\begin{table}[htbp]
\[\begin{array}{c|rrrrrrrr}
\text{type}&1&\ty A_1&\ty A_1^2&\ty A_1^3&\ty A_1^4&\ty A_2&\ty A_3&\ty D_4\\
\hline
|\mathcal U_s|&1&2&4&8&16&3&5&14
\end{array}\]
\end{table}
Thus the required bound holds.		

Next, suppose that $\C_{\bG^*}(s)^F/\C_{\bG^*}^\circ(s)^F$ is of order 2. 
Let $k=|{\mathcal U_s}^{\C_{\bG^*}(s)^F}|$, and $l$ be the number of $\C_{\bG^*}(s)^F$-orbits on $\mathcal U_s\setminus {\mathcal U_s}^{\C_{\bG^*}(s)^F}$.
Hence $|\mathcal U_s|=k+2l$.
By Lusztig's Jordan decomposition for disconnected groups and Clifford theory, we have that $|\cE(\bG^F,s)|=2k+l$ and $|\cE(\tbG^F,\tilde s)|=k+2l$.	
The possible types of $\C_{\bG^*}^\circ(s)$ are 1, $\ty A_1$, $\ty A_1^2$, $\ty A_3$, $\ty A_1^4$.
In particular, $|\cE(\tbG^F,\tilde s)|\le 16$.
If $|\mathcal U_s|\le 8$, then $|\cE(\bG^F,s)|\le16$.
Now we assume that $|\mathcal U_s|>8$, in which case $\C_{\bG^*}^\circ(s)$ is of type $\ty A_1^4$, and $F$ acts trivially on the factors.
Thus $|\mathcal U_s|=16$, $k=4$ and $l=6$. Hence $|\cE(\bG^F,s)|=14$.

Finally, suppose that $\C_{\bG^*}(s)^F/\C_{\bG^*}^\circ(s)^F$ is isomorphic to Klein four group. 
In particular, $\C_{\bG^*}(s)/\C_{\bG^*}^\circ(s)$ is isomorphic to Klein four group. 
The possible types of $\C_{\bG^*}^\circ(s)$ are 1, $\ty A_1$, $\ty A_1^4$.
In particular, $|\cE(\tbG^F,\tilde s)|\le 16$.
If $|\mathcal U_s|\le 4$, then $|\cE(\bG^F,s)|\le16$.
So we assume that $|\mathcal U_s|>4$, 	in which case $\C_{\bG^*}^\circ(s)$ is of type $\ty A_1^4$, and $F$ acts trivially on the factors.
Identify the unipotent characters of $\C_{\bG^*}^\circ(s)^F$ with \(\{1,\operatorname{St}\}^4\).
The action of \(\C_{\bG^*}(s)^F/\C_{\bG^*}^\circ(s)^F\) on \(\{1,\operatorname{St}\}^4\) has two invariant characters, three orbits with stabilizer of order 2, and two free orbits. 
Therefore, $|\cE(\bG^F,s)|= 2\cdot4+3\cdot2+2=16$.
Thus we complete the proof.
\end{proof}

Recall that $D$ denotes the subgroup of $\Aut(\bG^F)$ generated by the field and the graph automorphisms.
Now we generalize \cite[Prop.~4.4]{FLZ26}.

\begin{thm}\label{thm:cond}
Assume that $p$ is good for $\bG$ and $p$ does not divide $q$.
\begin{enumerate}[\rm(1)]
    \item There exists a blockwise $(\Lin_{p'}(\tbG^F/\bG^F)\rtimes D)$-equivariant bijection between $\cE(\tbG^F,p')$ and $\IBr(\tbG^F)$.
    \item Assume further that $p$ does not divide $|(\Z(\bG)/\Z^\circ(\bG))^F|$. Then there exists a blockwise $(\tbG^F\rtimes D)$-equivariant bijection between $\cE(\bG^F,p')$ and $\IBr(\bG^F)$.
\end{enumerate}

\end{thm}

\begin{proof}
(1) It suffices to show that for every $p$-block $\tB$ of $\tbG^F$, 
there exists a $(\Lin_{p'}(\tbG^F/\bG^F)\rtimes D)_{\tB}$-equivariant bijection between $\cE(\tbG^F,p')\cap\Irr(\tB)$ and $\IBr(\tB)$.
By \cite{BM89}, $\Irr(\tB)\subseteq\cE_p(\tbG^F,\tilde s)$ where $\tilde{s}$  is a semisimple $p'$-element of $\tbG^{*F}$.
In particular, $\cE(\tbG^F,p')\cap\Irr(\tB)\subseteq\cE(\tbG^F,\tilde s)$.

According to Theorem~\ref{thm:basic-set}, $\cE(\tbG^F,p')\cap\Irr(\tB)$ is a basic set for $\tB$.
If $\bG$ is of type $\ty A$, then by \cite{Di85a,Di85b,Ge91} the corresponding decomposition matrix is unitriangular, and we can get the equivariant bijection immediately.
In what follows, we assume that $\bG$ is not of type $\ty A$, and we fix a $p$-block $\tB$ of $\tbG^F$.

First assume that $(\Lin_{p'}(\tbG^F/\bG^F)\rtimes D)_{\tB}$ is hypoelementary.
Then by Corollary~\ref{cor:bas-set-equ}(1), there exists a $(\Lin_{p'}(\tbG^F/\bG^F)\rtimes D)_{\tB}$-equivariant bijection between $\cE(\tbG^F,p')\cap\Irr(\tB)$ and $\IBr(\tB)$.

Now we assume that $(\Lin_{p'}(\tbG^F/\bG^F)\rtimes D)_{\tB}$ is not hypoelementary.
Note that $(\Lin_{p'}(\tbG^F/\bG^F)\rtimes D)_{\tB}$ is isomorphic to a subgroup of $\Out(\bG^F)$.
Hence $\Out(\bG^F)$ is not hypoelementary, whence one of the following holds: $\bG$ is of type $\ty E_6$ and $|\Z(\bG^F)|=3$, or $\bG$ is of type $\ty D_4$ and $F$ is split (see \cite[Thm.~2.5.12]{GLS98}).

Suppose that $\bG$ is of type $\ty D_4$ and $F$ is split.
Then $|\cE(\tbG^F,\tilde s)|\le 16$ by Lemma~\ref{lem:D4-series}.
Then by Corollary~\ref{cor:bas-set-equ}(2), there exists a $(\Lin_{p'}(\tbG^F/\bG^F)\rtimes D)_{\tB}$-equivariant bijection between $\cE(\tbG^F,p')\cap\Irr(\tB)$ and $\IBr(\tB)$.

Finally we suppose that $\bG$ is of type $\ty E_6$ and $|\Z(\bG^F)|=3$.
If $\Lin_{p'}(\tbG^F/\bG^F)_{\tB}$ is trivial, then it can be checked directly that $(\Lin_{p'}(\tbG^F/\bG^F)\rtimes D)_{\tB}$ is hypoelementary.
Thus we have that $\Lin_{p'}(\tbG^F/\bG^F)_{\tB}$ is of order $3$. 
By Lemma~\ref{lem:E6-series}, $|\cE(\tbG^F,p')\cap\Irr(\tB)|\le 23$ or $|(\cE(\tbG^F,p')\cap\Irr(\tB))^{\Lin_{p'}(\tbG^F/\bG^F)_{\tB}}|\le 5$.
Thus Corollary~\ref{cor:bas-set-equ} yields a $(\Lin_{p'}(\tbG^F/\bG^F)\rtimes D)_{\tB}$-equivariant bijection between $\cE(\tbG^F,p')\cap\Irr(\tB)$ and $\IBr(\tB)$.

(2) The proof is somehow similar to (1).
It suffices to show that for every $p$-block $B$ of $\bG^F$, 
there exists a $(\tbG^F\rtimes D)_{B}$-equivariant bijection between $\cE(\bG^F,p')\cap\Irr(B)$ and $\IBr(B)$.
By \cite{BM89}, $\Irr(B)\subseteq\cE_p(\bG^F, s)$ where $s$ is a semisimple $p'$-element of $\bG^{*F}$.
In particular, $\cE(\bG^F,p')\cap\Irr(B)\subseteq\cE(\bG^F,s)$.

According to Theorem~\ref{thm:basic-set}, $\cE(\bG^F,p')\cap\Irr(B)$ is a basic set for $B$.
If $\bG$ is of type $\ty A$, then by \cite{Di85a,Di85b,Ge91} and \cite[Prop.~6.3]{FS23} the corresponding decomposition matrix is unitriangular, and we can get the equivariant bijection immediately.
In what follows, we assume that $\bG$ is not of type $\ty A$, and we fix a $p$-block $B$ of $\bG^F$.

Recall that $(\tbG^F\rtimes D)/\bG^F\Z(\tbG)^F$  is isomorphic to $\Out(\bG^F)$.
First assume that $(\tbG^F\rtimes D)_B/\bG^F\Z(\tbG)^F$ is hypoelementary. Then by Corollary~\ref{cor:bas-set-equ}(1), there exists a $(\tbG^F\rtimes D)_{B}$-equivariant bijection between $\cE(\bG^F,p')\cap\Irr(B)$ and $\IBr(B)$.
Hence we assume that $(\tbG^F\rtimes D)_B/\bG^F\Z(\tbG)^F$ is not hypoelementary, whence $\Out(\bG^F)$ is not hypoelementary. So one of the following holds: $\bG$ is of type $\ty E_6$ and $|\Z(\bG^F)|=3$, or $\bG$ is of type $\ty D_4$ and $F$ is split.

Suppose that $\bG$ is of type $\ty D_4$ and $F$ is split.
Then $|\cE(\bG^F,s)|\le 16$ by Lemma~\ref{lem:D4-series}.
Then by Corollary~\ref{cor:bas-set-equ}(2),  there exists a $(\tbG^F\rtimes D)_{B}$-equivariant bijection between $\cE(\bG^F,p')\cap\Irr(B)$ and $\IBr(B)$. 
Finally we suppose that $\bG$ is of type $\ty E_6$ and $|\Z(\bG^F)|=3$.
If $\tbG^F_{B}=\bG^F\Z(\tbG)^F$, then it can be checked directly that $(\tbG^F\rtimes D)_{B}/\bG^F\Z(\tbG)^F$ is hypoelementary.
If $\C_{\bG^*}(s)^F/\C_{\bG^*}^\circ(s)^F$ is trivial, then $\tbG^F$ acts trivially on $\cE(\bG^F,s)$, and since $D$ is hypoelementary, the assertion also holds by Corollary~\ref{cor:bas-set-equ}.
Thus we have that $\tbG^F_{B}/\bG^F\Z(\tbG)^F$ is of order $3$, that is, $\tbG^F_B=\tbG^F$, and we can assume that $\C_{\bG^*}(s)^F/\C_{\bG^*}^\circ(s)^F$ is nontrivial. 
By Lemma~\ref{lem:E6-series}, $|\cE(\bG^F,p')\cap\Irr(B)|\le 23$ or $|(\cE(\bG^F,p')\cap\Irr(B))^{\tbG^F}|\le 5$.
Thus Corollary~\ref{cor:bas-set-equ} yields a $(\tbG^F\rtimes D)_{B}$-equivariant bijection between $\cE(\bG^F,p')\cap\Irr(B)$ and $\IBr(B)$. 
\end{proof}

Now we establish the inductive Brauer--Glauberman condition for simple groups of Lie type at good non-defining primes.

\begin{thm}\label{thm:good-primes}
Suppose that $\bG$ is a simple algebraic group of simply-connected type over $\overline\FF_q$ and $F:\bG\to\bG$ is a Frobenius endomorphism endowing $\bG$ with an $\FF_q$-structure.   
Assume that $p$ is a prime good for $\bG$ and $p\nmid q$.
If $S:=\bG^F/\Z(\bG^F)$ is simple and $p\mid |S|$, then the inductive Brauer--Glauberman condition holds for $S$ at the prime $p$.
\end{thm}

\begin{proof}
By \cite[Cor.~4.8 and Thm.~5.8]{FS23}, we assume that $\bG$ is not of type $\ty A$, and $S$ has no exceptional Schur multiplier. Since finite simple groups with cyclic outer automorphism group satisfy the inductive Brauer--Glauberman condition, we may assume $S$ is not of Lie type $^3\ty{D}_4$.
To verify the inductive Brauer--Glauberman condition for $S$, it suffices to verify Hypothesis 5.5 of \cite{FS23}.
By \cite[Cor.~6.4]{FS23}, this can be transferred to the unitriangular assumption of decomposition matrices in Hypothesis 6.2 of \cite{FS23}.
In fact, the unitriangularity in Hypothesis 6.2 of \cite{FS23} is to prove conditions (ii), (iii) and (iv) of \cite[Thm.~4.9]{FS23}.
We note that (ii) and (iii) are just Brauer $A(\infty)$ condition.
In the proof of \cite[Thm.~5.7]{FS23}, the unitriangularity is used to prove the condition (iv) of \cite[Thm.~4.9]{FS23} ;
more precisely, what we need here is a $(\Lin_{p'}(\tbG^F/\bG^F)\rtimes D)$-equivariant bijection between $\cE(\tbG^F,p')$ and $\IBr(\tbG^F)$, where $\tbG$ and $F$ are defined as before.
Since $p$ is good, Theorem~\ref{thm:cond} and Theorem~\ref{mainthm-1} imply the  inductive Brauer--Glauberman condition.    
\end{proof}

In particular, combining \cite[Thm.~6.1]{FS23}, we complete the verification of the inductive Brauer--Glauberman condition for all simple groups of type $\ty C$ at any prime.

Finally, we prove Navarro's Conjecture~\ref{conj:coprime} for primes greater than~3.

\begin{proof}[Proof of Theorem~\ref{mainthm-2}]
In \cite[Thm.~A]{FR21}, part (ii) of Definition~6.1 of \cite{SV16} is verified.
We use the criterion from \cite{FS23} for the inductive Brauer--Glauberman condition.
By \cite[Thm.~5.11]{FS23}, we only need to consider simple groups of Lie type $\ty B_n$, $\ty C_n$, $\ty D_n$, $^2\ty D_n$, $\ty E_6$, $^2\ty E_6$, and $\ty E_7$, where $p$ is a non-defining characteristic.
Since $p\ge 5$, we have that $p$ is a good prime.
Therefore, we can complete the proof by Theorem~\ref{thm:good-primes}.
\end{proof}


\end{document}